\documentclass[10pt,a4paper]{article}

\usepackage{a4wide}
\usepackage{amsmath}
\usepackage{amsfonts}
\usepackage{amssymb}
\usepackage{amsthm}
\usepackage[all]{xy}
\SelectTips{cm}{10}
\usepackage{enumitem}
\usepackage{mathtools}
\usepackage{url}
\usepackage{rotating}

\theoremstyle{plain}
\newtheorem{thm}{Theorem}[section]
\newtheorem{cor}[thm]{Corollary}
\newtheorem{lem}[thm]{Lemma}
\newtheorem{prop}[thm]{Proposition}

\theoremstyle{definition}
\newtheorem{defn}[thm]{Definition}

\newtheorem{ex}[thm]{Example}

\theoremstyle{remark}
\newtheorem*{rem}{Remark}

\newcommand{\id}{\operatorname{id}}
\newcommand{\im}{\operatorname{im}}
\newcommand{\coker}{\operatorname{coker}}
\newcommand{\pr}{\operatorname{pr}}

\newcommand{\xra}[1]{\xrightarrow{#1}}

\newcommand{\ra}{\rightarrow}

\newcommand{\xlra}[1]{\xrightarrow{\ #1\ }}

\newcommand{\lra}{\longrightarrow}

\newcommand{\Ra}{\xRightarrow{\ \ }}

\newdir{c}{{}*!/-5pt/@^{(}}
\newdir{d}{{}*!/-5pt/@_{(}}
\newdir{ >}{{}*!/-5pt/@{>}}
\newdir{s}{{}*!/+10pt/@{}}
\newdir{|>}{{}*!/2pt/@{|}*@{>}}

\newcommand{\cof}[1][]{\mathbin{\:\!\!\xymatrix@1@C=15pt{{}\ar@{ >->}[r]^{#1} & {}}}}
\newcommand{\fib}[1][]{\mathbin{\:\!\!\xymatrix@1@C=15pt{{}\ar@{->>}[r]^{#1} & {}}}}
\newcommand{\embed}[1][]{\mathbin{\:\!\!\xymatrix@1@C=15pt{{}\ar@{c->}[r]^{#1} & {}}}}
\newcommand{\sect}[1]{\POS[l]+R*!!<0pt,\the\fontdimen22\textfont2>{\vphantom{|}}="a";[]+L*!!<0pt,\the\fontdimen22\textfont2>{\vphantom{|}} \ar@<-5pt>@/_2pt/"a"_-{#1}}

\DeclareMathAlphabet{\mathsfsl}{OT1}{cmss}{m}{sl}

\newcommand{\bbZ}{\mathbb Z}
\newcommand{\mcG}{\mathcal{G}}
\newcommand{\mcH}{\mathcal{H}}

\newcommand{\sSet}{\mathsf{sSet}}

\newcommand{\sfP}{\mathsfsl{P}}
\newcommand{\Gp}[2]{\mathbb{G}_{#1,#2}}
\newcommand{\Gen}[2]{\mathsfsl{Gen}_{#1,#2}}
\newcommand{\Gpfe}[2]{G_{#1,#2}^\mathrm{fe}}
\newcommand{\Gpse}[2]{G_{#1,#2}^\mathrm{se}}
\newcommand{\Map}{\mathsfsl{Map}}
\newcommand{\MPS}{\mathsfsl{MPS}}

\newcommand{\PMPS}{\mathsfsl{PMPS}}

\newcommand{\PairPMPS}{\mathsfsl{PairPMPS}}

\newcommand{\comma}{{},\,}

\newcommand{\pbsize}{15pt}
\newcommand{\pboffset}{.5}
\newcommand{\xycorner}[3]{\save #2="a";#1;"a"**{}?(\pboffset);"a"**\dir{-};#3;"a"**{}?(\pboffset);"a"**\dir{-}\restore}
\newcommand{\pb}{\xycorner{[]+<\pbsize,0pt>}{[]+<\pbsize,-\pbsize>}{[]+<0pt,-\pbsize>}}

\newcommand{\xymatrixc}[1]{\xy *!C\xybox{\xymatrix{#1}}\endxy}

\newlength{\hlp}

\newcommand{\orightbox}[2]{\settowidth{\hlp}{$#2$}\makebox[\hlp][r]{${#1}{#2}$}}
\newcommand{\leftbox}[2]{{}\phantom{#1} \save []+L*+<.5pc>!!<0pt,\the\fontdimen22\textfont2>!L{#1#2} \restore}
\newcommand{\rightbox}[2]{{}\phantom{#2} \save []+R*+<.5pc>!!<0pt,\the\fontdimen22\textfont2>!R{#1#2} \restore}

\newcommand{\bdry}{d}
\newcommand{\vertex}[1]{#1}

\newcommand{\them}{m}
\newcommand{\then}{n}

\newcommand{\theq}{q}

\newcommand{\thedim}{{n}}
\newcommand{\theotherdim}{{m}}
\newcommand{\thedimm}{{i}}

\newcommand{\stdsimp}[1]{\Delta^{#1}}
\newcommand{\horn}[2]{%
\mbox{$\xy
<0pt,-\the\fontdimen22\textfont2>;p+<.1em,0em>:
{\ar@{-}(0,0.1);(3,7)},
{\ar@{-}(3,7);(6,0.1)},
{\ar@{-}(3.2,7);(6.2,0.1)},
{\ar@{-}(3.4,7);(6.4,0.1)}
\endxy\;\!{}^{#1}_{#2}$}}

\newcommand{\Pnew}{{P_\thedim}}
\newcommand{\Pold}{{P_{\thedim-1}}}
\newcommand{\pin}{{\pi_\thedim}}

\newcommand{\oold}{{o}}

\newcommand{\Kn}{{K_{\thedim+1}}}

\newcommand{\Ln}{{L_\thedim}}

\newcommand{\kn}{{k_\thedim}}
\newcommand{\knp}{{k_\thedim'}}
\newcommand{\kip}{{k_\thedimm'}}
\newcommand{\knst}{{k_{\thedim*}}}

\newcommand{\pn}{{p_\thedim}}
\newcommand{\pnst}{{p_{\thedim*}}}

\newcommand{\jnst}{{j_*}}

\newcommand{\fn}{{f_\thedim}}
\newcommand{\varphin}{{\varphi_\thedim}}

\newcommand{\varphinst}{{\varphi_{\thedim*}}}
\newcommand{\alphan}{{\alpha_\thedim}}
\newcommand{\psin}{{\psi_\thedim}}

\newcommand{\clnull}[1]{\ref{cl:null}${}_{#1}$}
\newcommand{\clgen}[1]{\ref{cl:gen}${}_{#1}$}
\newcommand{\clpoly}[1]{\ref{cl:poly}${}_{#1}$}

\usepackage{titlesec}

\titleformat{\section}[block]
{\normalfont\Large\filcenter\bfseries}{\thesection.}{.33em}{}
\titlespacing*{\section}
{0pt}{3.5ex plus 1ex minus .2ex}{2.3ex plus .2ex}

\titleformat{\subsection}[runin]
{\normalfont\normalsize\bfseries}{\thesubsection.}{.33em}{}[.]
\titlespacing*{\subsection}
{0pt}{3.25ex plus 1ex minus .2ex}{.5em}

\usepackage{ifthen}
\usepackage[strict]{changepage}
\usepackage{mparhack}

\begin{document}

\author{M.\ Filakovsk\'{y}, L.\ Vok\v{r}\'{i}nek}
\title{Are two given maps homotopic?\\ An algorithmic viewpoint\footnote{%
The research of M.~F.\ was supported by Masaryk University project MUNI/A/0838/2012. The research of L.~V.\ was supported by the Center of Excellence -- Eduard \v{C}ech Institute (project P201/12/G028 of GA~\v{C}R).\newline
2010 \emph{Mathematics Subject Classification}. Primary 55Q05; Secondary 55P40.\newline
%55P40 Suspensions
%55Q05 Homotopy groups, general; sets of homotopy classes
\emph{Key words and phrases}. homotopy, suspension, polycyclic group.
}}
\date{\today}
\maketitle

\begin{abstract}
This paper presents two algorithms. In their simplest form, the first algorithm decides the existence of a pointed homotopy between given simplicial maps $f\comma g\colon X\to Y$ and the second computes the group $[\Sigma X,Y]^*$ of pointed homotopy classes of maps from a suspension; in both cases, the target $Y$ is assumed simply connected and the algorithms run in polynomial time when the dimension of $X$ is fixed. More generally, these algorithms work relative to $A\subseteq X$, fibrewise over a simply connected $B$ and also equivariantly when all spaces are equipped with a \emph{free} action of a fixed finite group $G$.
\end{abstract}

\section{Introduction}

In this paper, we are interested in decision algorithms for the existence of a homotopy between given maps $f\comma g\colon X\to Y$. For computational purposes, we assume $X$ and $Y$ given as finite simplicial complexes or more generally as finite simplicial sets, $f$ and $g$ as simplicial maps but we ask for a \emph{continuous} homotopy between them. It is well known that no homotopy decision algorithm may exist if $Y$ is allowed to be non-simply connected; this follows at once from Novikov's result \cite{Novikov} on the unsolvability of the word problem in groups. For this paper, we will thus restrict our attention to the case of a simply connected $Y$. In this respect, the following result is optimal. It is stated in a more general context of pointed homotopy.

{\renewcommand{\thethm}{A}\addtocounter{thm}{-1}
\begin{thm}\label{thm:main_simple}
There is an algorithm that decides the existence of a pointed homotopy between given simplicial maps $f\comma g\colon X\to Y$, where $X$, $Y$ are finite simplicial sets and $Y$ is assumed to be simply connected. When the dimension of $X$ is fixed, this algorithm runs in polynomial time.
\end{thm}}

In the paper \cite{cmk}, the authors gave an algorithmic solution to the following problem: given two simplicial sets $X$, $Y$, compute $[X,Y]$, i.e.\ the set of homotopy classes of \emph{continuous} maps from $X$ to $Y$. Their algorithm works under a certain connectivity restriction on $Y$. This restriction can be removed when the domain is replaced by a suspension -- this is our next result which, at the same time, generalizes the computation of homotopy groups of spaces described by Brown in \cite{Brown}.

{\renewcommand{\thethm}{B}\addtocounter{thm}{-1}
\begin{thm}\label{thm:suspension_simple}
There is an algorithm that computes the group $[\Sigma X,Y]^*$ of pointed homotopy classes of maps from a suspension $\Sigma X$ to a simply connected simplicial set $Y$. The running time of this algorithm is polynomial when the dimension of $X$ is fixed.
\end{thm}}

The group is presented on the output as a so-called \emph{fully-effective polycyclic group} -- this structure is introduced in Section~\ref{s:polycyclic} and allows one e.g.\ to compute a finite set of generators and relations and solve the word problem.

\subsection*{Fibrewise version}

In order to prove Theorem~\ref{thm:main_simple} in its full generality, we work fibrewise over $X$; the only non-fibrewise proof that we know of requires $g$ to be constant. At the same time, our proof uses heavily \cite{aslep} and thus, the above results can be easily extended to the case of spaces under $A$ and over $B$.

We denote by $A/\sSet/B$ the category of simplicial sets under $A$ and over $B$, i.e.\ simplicial sets $X$ equipped with a pair of maps $A\to X\to B$ whose composition is a fixed map $A\to B$, surpressed from the notation. Morphisms in this category are maps $f\colon X\to Y$ for which both triangles in
\begin{equation}\label{e:A/sSet/B}
\xymatrixc{
A \ar[r]^-\alpha \ar[d]_-\iota & Y \ar[d]^-\varphi \\
X \ar[r]_-\beta \ar[ru] \POS?<>(.5)+/:a(90).5pc/*{\scriptstyle f} & B
}\end{equation}
commute. There is also an obvious notion of homotopy (relative to $A$ and fibrewise over $B$). In case $\iota$ is an inclusion and $\varphi$ is a Kan fibration, the resulting set of homotopy classes will be denoted by $[X,Y]^A_B$. For general $X\comma Y\in A/\sSet/B$, we define $[X,Y]^A_B$ by first replacing $\iota$ up to weak homotopy equivalence by an inclusion $A\cof X^\mathrm{cof}$ and $\varphi$ by a Kan fibration $Y^\mathrm{fib}\fib B$ and then setting $[X,Y]^A_B=[X^\mathrm{cof},Y^\mathrm{fib}]^A_B$.

For the fibrewise version of Theorem~\ref{thm:suspension_simple}, we need to generalize the notions of pointed spaces and suspensions. We say that a space $\varphi\colon Y\to B$ over $B$ is \emph{pointed} if there is provided a section $o\colon B\to Y$ of $\varphi$. For any space $\beta\colon X\to B$ over $B$, the composition $X\xra\beta B\xra o Y$ will be also denoted by $o$ and called the \emph{zero map}. If $\alpha=o$ in \eqref{e:A/sSet/B} and $\iota$ is injective, $\varphi$ a Kan fibration, then $[X,Y]^A_B$ is the set of homotopy classes of maps $f\colon X\to Y$ over $B$ that are zero on $A$.

The \emph{fibrewise suspension} $\Sigma_BX$ is obtained from the cylinder $I\times X$ by separately squashing each of $0\times X$ and $1\times X$ to $B$ using the given projection $\beta\colon X\to B$; it is naturally a space over $B$. The map $\iota\colon A\to X$ induces a map $\Sigma_BA\to\Sigma_BX$.

{\renewcommand{\thethm}{C}\addtocounter{thm}{-1}
\begin{thm}\label{thm:main}
Let a commutative square
\[\xymatrix{
A \ar[r]^-\alpha \ar[d]_-\iota & Y \ar[d]^-\varphi \\
X \ar[r]_-\beta & B
}\]
be given on the input, where all spaces are finite simplicial sets, both $Y$ and $B$ simply connected. Then the following algorithms exist:
\begin{enumerate}[topsep=2pt,itemsep=2pt,parsep=2pt,labelindent=.5em,leftmargin=*,label=\bfseries\upshape\thethm.\arabic*.]
\renewcommand{\theenumi}{\thethm.\arabic{enumi}}
\renewcommand{\labelenumi}{\bfseries\upshape\thethm.\arabic{enumi}.}
\item\label{thm:main_pta}
Given two maps $f \comma g\colon X\to Y$ in $A/\sSet/B$, decide whether they represent the same element in $[X,Y]^A_B$.
\renewcommand{\theenumi}{\thethm.\arabic{enumi}}
\renewcommand{\labelenumi}{\bfseries\upshape\thethm.\arabic{enumi}.}
\item\label{thm:main_ptb}
Given a zero section $o\colon B\to Y$, compute the group $[\Sigma_BX,Y]^{\Sigma_BA}_B$ of maps $\Sigma_BX\to Y$ over $B$ that are \emph{zero} on $\Sigma_BA$.
\end{enumerate}
When the dimensions of $A$ and $X$ are fixed, these algorithms run in polynomial time.
\end{thm}}

Theorems~\ref{thm:main_simple}~and~\ref{thm:suspension_simple} are obtained from Theorem~\ref{thm:main} by setting $A=*$ and $B=*$.

We remark that \cite{aslep} also covers the possibility that all spaces are equipped with a free action of a fixed finite group $G$ and all maps and homotopies are required to be $G$-equivariant. This is also the case here but we have decided not to complicate the statement even further. We believe that an interested reader may fill in details easily.

\subsection*{Notation}

We denote the \emph{standard $n$-simplex} by $\stdsimp n$, its $i$-th \emph{vertex} by $i$, its $i$-th \emph{face} by $d_i\stdsimp n$ and its \emph{boundary} by $\partial\stdsimp n$. The $i$-th \emph{horn} in $\stdsimp n$, i.e.\ the simplicial subset spanned by the faces $\bdry_j\stdsimp n$, $j\neq i$, will be denoted $\horn{n}{i}$. For simplicity, we will also denote $I=\stdsimp 1$. Then $\partial I^\theq$ is the obvious boundary of the \emph{$\theq$-cube}, i.e.\ of the $\theq$-fold product $I^\theq=I\times\cdots\times I$.

\section{Moore--Postnikov towers}\label{sec:Moore_Postnikov}

The proof of Theorem~\ref{thm:main} relies on computations in the Moore--Postnikov tower of $Y$ over $B$. The tower has been constructed in \cite{polypost,aslep}. Here we only give a brief summary of the main results concerned with the construction and computations in the tower.

\subsection*{Definition of the Moore--Postnikov tower}

Let $\varphi\colon Y\to B$ be a map. A (simplicial)
\emph{Moore--Postnikov tower} for $\varphi$ is a commutative diagram
\[\xymatrix{
& & {} \ar@{.}[d] \\
& & \Pnew \ar[d]^-{\pn} \ar@/^30pt/[ddd]^-{\psi_\thedim} \\
& & \Pold \ar@{.}[d] \\
Y \ar[uurr]^<>(.6){\varphi_\thedim} \ar[urr]_<>(.6){\varphi_{\thedim-1}} \ar[rr]_<>(.6){\varphi_1} \ar[drr]_<>(.6){\orightbox{\scriptstyle\varphi={}}{\scriptstyle\varphi_0}}
& & P_{1} \ar[d]^-{p_1} \\
& & \leftbox{P_{0}}{{}=B}
}\]
satisfying the following conditions:
\begin{itemize}[topsep=2pt,itemsep=2pt,parsep=2pt,leftmargin=\parindent]
\item
The induced map $\varphi_{\thedim*}\colon \pi_\thedimm(Y)\to \pi_{\thedimm}(\Pnew)$ is an isomorphism for $0\leq\thedimm\leq\thedim$ and an epimorphism for $\thedimm=\thedim+1$.

\item
The induced map $\psi_{\thedim*}\colon\pi_\thedimm(\Pnew)\to\pi_\thedimm(B)$ is an isomorphism for $\thedimm\ge\thedim+2$ and a monomorphism for $\thedimm=\thedim+1$.

\item
There exists a pullback square
\[\xymatrix{
\Pnew \pb \ar[r] \ar[d]_-\pn & E(\pin,\thedim) \ar[d]^{\delta} \\
\Pold \ar[r]_-\knp & K(\pin,\thedim+1)
}\]
identifying $\Pnew$ with the pullback $\Pold\times_{K(\pin,\thedim+1)}E(\pin,\thedim)$. Here, $K(\pi_n,n+1)$ is the Eilenberg--MacLane space and $E(\pi_n,n)$ its path space. These have standard simplicial models with $K(\pi_n,n+1)$ a minimal complex and $\delta$ a minimal fibration, see~\cite{may}.
\end{itemize}

From the computational perspective, the Moore--Postnikov tower faces the following problem: the standard simplicial models for Eilenberg--MacLane spaces, although minimal, are often infinite. This is solved by a somewhat technical notion of a simplicial set with effective homology that was introduced by Sergeraert etal. A detailed exposition is given in \cite{SergerGenova} and an extension to free actions of a finite group $G$ is described in \cite{aslep}. We will not need an explicit definition here -- the main property for us will be that all simplices have a well defined representation in a computer. Thus, for example, a simplicial map $X\to P_n$ is given by a finite amount of data. We also recall that a map is said to be \emph{computable} if an algorithm is provided that evaluates this map at a given element.

We have the following theorem, whose non-fibrewise (i.e.\ with $B=*$) version is explained in much more detail in \cite{polypost}.

\begin{thm}[{\cite[Theorem~3.2]{aslep}}]\label{t:MP_tower}
There is an algorithm that, given a map $\varphi\colon Y\to B$ between finite simply connected simplicial sets and an integer $\thedim$, constructs the first $\thedim$ stages of a Moore--Postnikov tower for $\varphi$. The stages $P_\thedimm$ are constructed as simplicial sets with effective homology, and $\varphi_\thedimm$, $\kip$, $p_\thedimm$ as computable maps.\qed
\end{thm}

From now on, we will assume that $\iota$ is an inclusion; if this was not the case, simply replace the space $X$ in the square \eqref{e:A/sSet/B} by the mapping cylinder of $\iota$, i.e.\ the space $X^\mathrm{cof}=(I\times A)\cup_\iota X$.

\begin{thm}[{\cite[Theorem~3.3]{aslep}}]\label{t:n_equivalence}
The map $\varphin \colon Y\to\Pnew$ induces a bijection $\varphinst\colon[X,Y]^A_B \to [X,\Pnew]^A_B$ for every $\thedim$-dimensional simplicial set $X$.\qed
\end{thm}

This theorem allows us to replace the square \eqref{e:A/sSet/B} by
\[\xymatrix{
A \ar[r]^-\alphan \ar@{ >->}[d]_-\iota & \Pnew \ar@{->>}[d]^-\psin \\
X \ar[r]_-\beta \ar[ru] \POS?<>(.5)+/:a(90).5pc/*{\scriptstyle\fn} & B
}\]
in which $\alphan=\varphin\alpha$ and $\fn=\varphin f$. Since $\psin$ is a Kan fibration, the homotopy classes in $[X,\Pnew]^A_B$ are represented by simplicial maps $X\to\Pnew$ under $A$ and over $B$ (no replacements needed).

\subsection*{Computations with Moore--Postnikov towers}

For our algorithm, it will be essential to lift homotopies. Moreover, homotopy concatenation will serve as the main tool in the computations with maps defined on suspensions. The proofs of the results in this subsection can be found \cite{aslep}. We start with a general algorithm for lifting maps by one stage.

\begin{prop}[{\cite[Proposition~3.5]{aslep}}] \label{prop:lift_ext_one_stage}
There is an algorithm that, given a diagram
\[\xymatrix{
A \ar[r] \ar@{ >->}[d] & \Pnew \ar@{->>}[d]^-{\pn} \\
X  \ar[r] \ar@{-->}[ru] & \Pold
}\]
decides whether a diagonal exists. If it does, it computes one.
\end{prop}

The following two special cases apply even to lifting through multiple stages.

\begin{prop}[homotopy lifting, {\cite[Proposition~3.6]{aslep}}] \label{prop:homotopy_lifting}
Given a diagram
\[\xymatrix{
(\vertex\thedimm\times X)\cup(I\times A) \ar[r] \ar@{ >->}[d]_-\sim & \Pnew \ar@{->>}[d] \\
\stdsimp{1}\times X  \ar[r] \ar@{-->}[ru] & P_\theotherdim
}\]
where $\thedimm\in\{0,1\}$, it is possible to compute a diagonal. In other words, one may lift homotopies in Moore--Postnikov towers algorithmically.
\end{prop}

The second special case will be used in Section~\ref{s:concatenation} to concatenate homotopies.

\begin{prop}[homotopy concatenation, {\cite[Proposition~3.7]{aslep}}] \label{prop:homotopy_concatenation}
Given a diagram
\[\xymatrix{
(\horn{2}{i}\times X)\cup(\stdsimp{2}\times A) \ar[r] \ar@{ >->}[d]_-\sim & \Pnew \ar@{->>}[d] \\
\stdsimp{2}\times X  \ar[r] \ar@{-->}[ru] & P_\theotherdim
}\]
where $\thedimm\in\{0,1,2\}$, it is possible to compute a diagonal. In other words, one may concatenate homotopies in Moore--Postnikov towers algorithmically.
\end{prop}

\section{Maps out of suspensions I}

\subsection*{Pointed fibrations}

From now on, we will assume that $\psin\colon\Pnew\to B$ is equipped with a zero section $o\colon B\to\Pnew$. Further, we will assume that $\alpha_n=o$, i.e.\ $[X,\Pnew]^A_B$ will now denote the set of homotopy classes of maps $f\colon X\to \Pnew$ that are over $B$ and \emph{zero} on $A$.

\subsection*{Homotopy concatenation}\label{s:concatenation}
We will now use Proposition~\ref{prop:homotopy_concatenation} to make $[\Sigma_BX,\Pnew]^{\Sigma_BA}_B$ into a group. It is simple to see that this set is isomorphic to $[I\times X,\Pnew]^{(\partial I\times X)\cup(I\times A)}_B$. We will work with the second description and represent the elements of $[\Sigma_BX,\Pnew]^{\Sigma_BA}_B$ by fibrewise homotopies $I\times X\to\Pnew$, starting and finishing at the zero map and zero on $I\times A$.

Let $h_2\comma h_0\colon I\times X\to\Pnew$ be two such homotopies. Viewing each $h_i$ as defined on $d_i\stdsimp 2\times X$, we obtain a single map $\horn 21\times X\to\Pnew$ which, together with the zero map $o\colon\stdsimp 2\times A\to\Pnew$, prescribes the top map in Proposition~\ref{prop:homotopy_concatenation}. The bottom map is the composition $\stdsimp 2\times X\xra{\pr}X\xra{\beta}B$, i.e.\ we take $\theotherdim=0$. Let $\stdsimp 2\times X\to\Pnew$ be the diagonal map computed by Proposition~\ref{prop:homotopy_concatenation}. Then we will call its restriction to $\bdry_1\stdsimp 2\times X$ the \emph{concatenation} of $h_2$ and $h_0$ and denote it by $h_0+h_2$. The inverse of a homotopy is computed similarly. The situation is summarized in the following subsection.

\subsection*{Semi-effective groups}

In our setting, a group $G$ is represented by a set $\mcG$, whose elements are called \emph{representatives}; we also assume that the representatives can be stored in a computer. For $\gamma\in\mcG$, let $[\gamma]$ denote the element of $G$ represented by $\gamma$. The representation is generally non-unique -- we may have $[\gamma]=[\delta]$ for $\gamma\ne\delta$. We will write our groups additively.

\begin{defn}
We call $G$ represented in the above way \emph{semi-effective}, if algorithms for the following three tasks are available:
\begin{itemize}[topsep=2pt,itemsep=2pt,parsep=2pt,leftmargin=\parindent]
\item
	provide an element $o\in\mcG$ with $[o]=0$ (the neutral element);
\item
	given $\gamma\comma\delta\in\mcG$, compute $\varepsilon\in\mcG$ with $[\varepsilon]=[\gamma]+[\delta]$;
\item
	given $\gamma\in\mcG$, compute $\delta\in\mcG$ with $[\delta]=-[\gamma]$.
\end{itemize}
\end{defn}

An important example of a semi-effective group is the cohomology group $H^n(X,A;\pi)$. It is represented by maps $X\to K(\pi,n)$ that are zero on $A$. For the minimal model of $K(\pi,n)$ that we use throughout the paper, such maps are in a bijective correspondence with cocycles $Z^n(X,A;\pi)$, see \cite{may}. In this case, much more is true: since $Z^n(X,A;\pi)$ is finitely generated abelian, it is possible to decide whether a given element $[\gamma]$ is an integral combination of $[\gamma_1],\ldots,[\gamma_r]$; if this is the case, the coefficients $z_1,\ldots,z_r$ in the expression $[\gamma]=z_1[\gamma_1]+\cdots+z_r[\gamma_r]$ are computable too. Later, we will formalize this in the notion of a fully effective abelian group. Returning to the suspension, we have already obtained the following result.

\begin{prop}\label{p:semi_effective}
The set $[\Sigma_BX,\Pnew]^{\Sigma_BA}_B\cong[I\times X,\Pnew]^{(\partial I\times X)\cup(I\times A)}_B$ is a semi-effective group rep\-re\-sent\-ed by the set of all simplicial maps $I\times X\to\Pnew$ over $B$ that are zero on $(\partial I\times X)\cup(I\times A)$.\qed
\end{prop}

\section{Deciding the existence of a homotopy}

\subsection*{An exact sequence associated with a fibration}

We start with the following notation: $K_{n+1}=B\times K(\pi_n,n+1)$ and $L_n=B\times K(\pi_n,n)$. There are maps
\[\Ln\xlra{j}\Pnew\xlra{\pn}\Pold\xlra{\kn}\Kn,\]
where $\kn$ is the ``fibrewise Postnikov invariant'' $(\psi_{\thedim-1},k_n')\colon\Pold\lra B\times K(\pi_n,n+1)$ and $j$ is the following embedding: writing $o=(o',o'')\colon B\to\Pnew\subseteq\Pold\times E(\pi_n,n)$, it is defined as $j(b,z)=(o'(b),o''(b)+z)$. The image of $j$ clearly consists precisely of those simplices of $\Pnew$ that map to the zero section $o'$ of $\Pold$. The following sequence of pointed sets is \emph{exact} by \cite[Theorem~4.8]{aslep} (the relevant parts of the proof do not use the stability assumption $n\leq 2d$):
\begin{equation}\label{e:les}
 [\Sigma_BX,\Pold]^{\Sigma_BA}_B \xlra{\partial}[X,\Ln]^A _B \xlra{\jnst}[X,\Pnew]^A _B \xlra{\pnst}[X,\Pold]^A _B\xlra{\knst}[X,\Kn]^A _B.
\end{equation}

The isomorphisms $[X,\Ln]^A _B\cong H^n(X,A;\pi_n)$ and $[X,\Kn]^A _B\cong H^{n+1}(X,A;\pi_n)$ show that these sets are abelian groups that can be computed easily. The group homomorphism $\partial$ is defined in the following way. Given a homotopy $h\colon I \times X \to P_{n-1}$, lift it to a homotopy $\widetilde h\colon I \times X \to P_{n}$ in such a way that $(\vertex 0 \times X) \cup (I \times A)$ maps to the zero section, using Proposition~\ref{prop:homotopy_lifting}. Since the restriction of $\widetilde h$ to $\vertex 1\times X$ takes values in the image of $j$, it could be interpreted as a map $X\to\Ln$. This map is then a representative of $\partial[h]$.

\subsection*{Proof of Theorem~\ref{thm:main_pta}}

We will prove Theorem~\ref{thm:main} by induction. First, we list a series of claims:

\begin{enumerate}[topsep=2pt,itemsep=2pt,parsep=2pt,labelindent=0.5em,leftmargin=*,label=\textbf{(null)${}_n$}]
\renewcommand{\theenumi}{\textbf{(gen)}}
\renewcommand{\labelenumi}{\textbf{(gen)${}_n$}}
\item\label{cl:gen}
	It is possible to compute a finite set of generators of $[I\times X,\Pnew]^{(\partial I\times X)\cup(I\times A)}_B$.
\renewcommand{\theenumi}{\textbf{(null)}}
\renewcommand{\labelenumi}{\textbf{(null)${}_n$}}
\item\label{cl:null}
	It is possible to decide whether a given map $f\colon X\to\Pnew$ under $A$ over $B$ and is nullhomotopic; when this is the case, it is possible to compute a nullhomotopy, i.e.\ a homotopy from the zero map to $f$.
\end{enumerate}

\begin{proof}[Proof of Theorem~\ref{thm:main_pta} from \clnull{n}]
Let $n=\dim X$. Since $[X,Y]^A_B\cong[X,\Pnew]^A_B$ by Theorem~\ref{t:n_equivalence}, it is enough to decide whether the corresponding maps $f_n\comma g_n\colon X\to\Pnew$ are homotopic. Taking the pullback of $\psin\colon\Pnew\to B$ along $\beta\colon X\to B$ yields another Moore--Postnikov stage (see \cite[Section~4.10]{aslep}), this time with a section $o=(\id,g_n)$, as in the following diagram:
\[
\xymatrix{
A\ar[r]^-{(\iota,\alpha_n)} \ar@{ >->}[d]_-\iota	 & X \times_{B} P_n  \ar@{->>}[d] \\
X\ar[r]_-{\id} \ar[ur] \POS?<>(.45)+/:a(90).5pc/*{\rotatebox{35}{$\scriptstyle(\id,f_n)$}} & X \ar@/_10pt/[u]_-{(\id,g_n)}
}
\]
Thus, according to \clnull{n}, it is possible to decide whether $(\id,f_n)$ is nullhomotopic, i.e.\ homotopic to $(\id,g_n)$. Clearly, this is equivalent to $f_n$ being homotopic to $g_n$.
\end{proof}

The claim \clnull{n} is proved by induction using \clgen{n-1}. This is essentially contained in \cite[Section~4.9]{aslep}; we reproduce the algorithm here for reader's convenience but omit the proof of correctness.

\begin{proof}[Proof of \clnull{n-1}${}+{}$\clgen{n-1}${}\Ra{}$\clnull{n}]
First, we compute a nullhomotopy $h'$ of the composition $p_n f\colon X\to\Pold$ by \clnull{n-1}. Next, we lift this nullhomotopy using Proposition~\ref{prop:homotopy_lifting} to a homotopy $\widetilde h'\colon f' \sim f$. Since $p_n f' = o$, we interpret $f'$ as a map $f'\colon X\to\Ln$. We use \clgen{n-1} to decide whether $[f'] \in \im \partial$ and further to compute $h''$ with $\partial[h'']= [f']$. Using Proposition~\ref{prop:lift_ext_one_stage}, it is possible to compute a lift $\widetilde h''$ that starts at the zero map and finishes at $f'$. Thus, the concatenation $h = \widetilde h'+\widetilde h''$, computed by Proposition~\ref{prop:homotopy_concatenation}, is a homotopy from the zero map to $f$. If either of $h'$, $h''$ fails to exist, the map $f$ is not nullhomotopic.
\end{proof}

Thus, it remains to prove \clgen{n}. To make the induction possible, we will have to strengthen the claim and compute more than just generators, namely the structure of a fully effective polycyclic group.

\section{Polycyclic groups}\label{s:polycyclic}

\subsection*{Fully effective abelian groups}

First, we recall from \cite{cmk} some basic computational aspects of abelian groups.

\begin{defn}
Let $G$, $H$ be semi-effective groups with sets of representatives $\mathcal{G}$, $\mathcal{H}$. A homomorphism $f\colon G \to H$ will be called a \emph{computable homomorphism} if there exists a computable mapping $\varphi\colon \mathcal{G} \to \mathcal{H}$ such that $f([\gamma])=[\varphi(\gamma)]$, i.e.\ if there is provided an algorithm that computes a representative of $f(g)$ from each representative of $g$.
\end{defn}

Given a semi-effective group, we would like to obtain some further information about it, e.g.\ compute a finite list of generators or solve the word problem. For abelian groups, this is accomplished easily with the help of the classification of finitely generated abelian groups:

\begin{defn}
We call a semi-effective abelian group $G$ \emph{fully effective} if there is given an isomorphism $G\cong\bbZ/q_1\oplus\cdots\oplus\bbZ/q_r$, computable together with its inverse. In detail, this consists of
\begin{itemize}[topsep=2pt,itemsep=2pt,parsep=2pt,leftmargin=\parindent]
\item
	a finite list of generators $g_1,\ldots,g_r$ of $G$ (given by representatives) and their orders $q_1,\ldots,q_r\in\{2,3,\ldots\}\cup\{0\}$ (where $q_\thedimm=0$ gives $\bbZ/q_\thedimm=\bbZ$),
\item
	an algorithm that, given $\gamma\in\mcG$, computes integers $z_1,\ldots,z_r$ so that $[\gamma]=z_1 g_1 + \cdots + z_r g_r$; each coefficient $z_i$ is unique within $\bbZ/q_i$.
\end{itemize}
\end{defn}

As explained, $H^n(X,A;\pi)$ is fully effective when represented by maps $X\to\Ln$ over $B$ that are zero on $A$. This is provided by a Smith normal form algorithm, see \cite[Lemma~4.6]{aslep}.

\begin{lem}[kernel and cokernel, {\cite[Lemmas~2.2~and~2.3]{cmk}}]\label{l:ker_coker}
Let $f\colon G\to H$ be a computable homomorphism of fully effective abelian groups. Then both $\ker f$ and $\coker f$ can be represented as fully effective abelian groups. More generally, the computation of $\coker f$ only requires $H$ fully effective abelian, a list of generators of $G$ (not necessarily abelian) and $f$ computable.
\end{lem}

Another useful construction is \cite[Lemma~2.4]{cmk} that shows that the class of fully effective abelian groups is closed under extensions. We will not use this result; instead, we will need its generalization to the case of polycyclic groups, namely Proposition~\ref{p:poly_extension}.

\subsection*{Polycyclic groups}

The group $[I\times X, Y]^{(\partial I\times X) \cup (I\times A)}_B$ is not abelian and we will thus need to extend some of the machinery from abelian groups to a wider class of groups, called polycyclic.

\begin{defn}\label{polycykl}
A group $G$ is called \emph{polycyclic}, if it has a subnormal series with cyclic factors. 
In detail, there exists a sequence of subgroups 
\begin{equation}\label{e:subnormal}
G = G_r \geq G_{r-1} \geq \cdots \geq G_1 \geq G_0 = 0
\end{equation}
such that:
\begin{itemize}[topsep=2pt,itemsep=2pt,parsep=2pt,leftmargin=\parindent]
\item
	$G_{i-1}$ is a normal subgroup of $G_{i}$ for $i = 1, \ldots, r$,
\item
	$G_{i} / G_{i-1}$ is a cyclic group for $i = 1, \ldots, r$.
\end{itemize}
\end{defn}

\begin{ex}
Every finitely generated abelian group is polycyclic: when $G\cong\bbZ/q_1\oplus\cdots\oplus\bbZ/q_r$ with the corresponding generators $g_1,\ldots,g_r$, the filtration is given by $G_i=[g_1,\ldots,g_i]$, i.e.\ the subgroup generated by $g_1,\ldots,g_i$.
\end{ex}

Suppose that elements $g_i\in G_i$ have been chosen in such a way that their images in $G_i/G_{i-1}$ are generators of these cyclic groups (clearly, such a choice is possible). Denoting by $q_i$ the order of $G_i/G_{i-1}$, the following map
\begin{align*}
\bbZ/q_1 \times \cdots \times \bbZ/q_r & \lra G \\
(z_1,\ldots,z_r) & \longmapsto z_1 g_1 + \cdots + z_r g_r
\end{align*}
is easily seen to be bijective: given $g\in G$, consider its image $z_r\in G_r/G_{r-1}\cong\bbZ/q_r$. Then $g-z_rg_r\in G_{r-1}$ and we continue in the same manner to show that $g-z_rg_r-\cdots-z_1g_1\in G_0=0$, i.e.\ $g=z_1g_1+\cdots+z_rg_r$ in a unique way. In particular, $G$ is generated by $g_1,\ldots,g_r$. At the same time, the \emph{word problem} in $G$, i.e.\ the problem of deciding whether two given words in the generators $g_i$ are equal, can be translated to $\bbZ/q_1 \times \cdots \times \bbZ/q_r$ and easily solved there. This leads to our notion of a fully effective polycyclic group.

\begin{defn}\label{d:poly}
We say that a semi-effective group $G$, represented by a set $\mcG$, is \emph{fully effective polycyclic} if it is polycyclic with subnormal series \eqref{e:subnormal} and a bijection $\bbZ/q_1 \times \cdots \times \bbZ/q_r \cong G$ as above is computable together with its inverse. In detail, this consists of
\begin{itemize}[topsep=2pt,itemsep=2pt,parsep=2pt,leftmargin=\parindent]
\item
	a finite list of elements $g_1\in G_1,\ldots,g_r\in G_r$ (given by representatives) and the orders $q_1,\ldots,q_r\in\{2,3,\ldots\}\cup\{0\}$ of $G_\thedimm/G_{\thedimm-1}$ (where $q_\thedimm=0$ gives $\bbZ/q_\thedimm=\bbZ$),
\item
	an algorithm that, given $\gamma\in\mcG$, computes integers $z_1,\ldots,z_r$ so that $[\gamma]=z_1 g_1 + \cdots + z_r g_r$; each coefficient $z_i$ is unique within $\bbZ/q_i$.
\end{itemize}
\end{defn}

As explained just prior to the definition, the algorithm in the second point is equivalent to the computability of the projections $p_i\colon G_i\to G_i/G_{i-1}\cong\bbZ/q_i$.

\begin{rem}
In fact, it is even possible to specify (the isomorphism type of) the whole group by a finite amount of data. This includes the conjugation action $g_i+g_j-g_i\in G_{i-1}$ for $i>j$ and the multiples $q_ig_i\in G_{i-1}$.
\end{rem}

\subsection*{Computations with fully effective polycyclic groups}

Next, we show that fully effective polycyclic groups are closed under kernels and extensions.

\begin{prop} \label{p:poly_ker}
Let $G$ be a fully effective polycyclic group, $H$ a fully effective abelian group and $f\colon G \to H$ a computable homomorphism.
Then it is possible to compute $K = \ker f$ as a fully effective polycyclic group.
\end{prop}
\begin{proof}
We will proceed by induction with respect to the length $r$ of the subnormal series for $G$. We denote $K_i=\ker f|_{G_i}=G_i\cap K$. In the following diagram, every row is a short exact sequence and so are the solid columns.
\[
\xymatrix{
&
	0 \ar[d] &
	0 \ar[d] &
	0 \ar@{-->}[d] \\
0 \ar[r] &
	K_{r-1} \ar@{c->}[d] \ar@{c->}[r] &
	K_r \ar@{c->}[d] \ar[r] &
	K_r/K_{r-1} \ar[r] \ar@{c-->}[d] &
	0 \\
0 \ar[r] &
	G_{r-1} \ar@{c->}[r] \ar[d]^-f &
	G_r \ar[r] \ar[d]^-f &
	G_r/G_{r-1} \ar[r] \ar@{-->}[d]^-{f'} &
	0 \\
0 \ar[r] &
	f(G_{r-1}) \ar@{c->}[r] \ar[d] &
	f(G_r) \ar[r] \ar[d] &
	f(G_r)/f(G_{r-1}) \ar[r] \ar@{-->}[d] &
	0 \\
&
	0 &
	0 &
	0
}
\]
It is easy to see that the dashed column is then also exact. By induction, $K_{r-1}$ is fully effective polycyclic. By Lemma~\ref{l:ker_coker}, it is possible to compute $\ker f'\cong K_r/K_{r-1}$; say that it is generated by $t_r\in G_r/G_{r-1}\cong\bbZ/q_r$. This means that $f(t_rg_r)\in f(G_{r-1})$ and thus, from the knowledge of the generators of $G_{r-1}$, it is possible to compute some $h\in G_{r-1}$ with $f(t_rg_r)=f(h)$. Finally, $-h+t_rg_r\in K_r$ is the required element mapping to the generator $t_r\in K_r/K_{r-1}$. The projection $K_r\to K_r/K_{r-1}\cong\bbZ/(q_rt_r^{-1})$ is the composition
\[\xymatrix{
K_r \ar@{c->}[r] & G_r \ar[r] & G_r/G_{r-1}\cong\bbZ/q_r \ar@{-->}[r]^-{t_r^{-1}\times} & \bbZ/(q_rt_r^{-1})
}\]
(the multiplication by $t_r^{-1}$ is defined on the image of $K_r$) and is thus computable.
\end{proof}

The following corollary states that we can further compute kernels of computable maps between fully effective polycyclic groups.
\begin{cor}
Let $G$, $H$ be fully effective polycyclic groups and $f\colon G \to H$ a computable homomorphism. Then it is possible to compute $K = \ker f$ as a fully effective polycyclic group.
\end{cor}

\begin{proof}
Suppose that $H$ has a subnormal series of length $s$. We set $K_j=f^{-1}(H_j)$ and observe that $K=K_0$. We may compute inductively $K_{j-1}$ from $K_j$ using Proposition~\ref{p:poly_ker} as the kernel of the composition $K_j\xlra f H_j\lra H_j/H_{j-1}$ with abelian codomain.
\end{proof}

\begin{rem}
It is also possible to compute cokernels, see \cite{polycgroup}. However, we do not see a way of controlling the running time of such an algorithm.
\end{rem}

\begin{prop}\label{p:poly_extension}
Suppose that there is given a short exact sequence of semi-effective groups
\[
\xymatrix{
0 \ar[r] & K \ar[r]_-f & G \ar[r]_-g  \sect{t} & H \ar[r] \sect{\sigma} & 0
}
\] 
with $K$, $H$ fully effective polycyclic, $f$, $g$ computable homomorphisms, $t\colon\im f\to K$ a computable inverse of $f$ and $\sigma\colon\mcH\to\mcG$ a computable mapping such that $g[\sigma(\eta)]=[\eta]$. Then there is an algorithm that equips $G$ with a structure of a fully effective polycyclic group.
\end{prop}

\begin{proof}
We have the following filtration
\[
G = g^{-1}(H_s)\geq g^{-1}(H_{s-1}) \geq \cdots \geq g^{-1}(H_0) = f(K_r)\geq f(K_{r-1}) \geq \cdots \geq f(K_0) = 0
\]
with filtration quotients either $K_i/K_{i-1}$ or $H_j/H_{j-1}$, the corresponding projections
\begin{align*}
& f(K_i)\xlra t K_i\lra K_i/K_{i-1}, \\
& g^{-1}(H_j)\xlra g H_j\lra H_j/H_{j-1}
\end{align*}
and generators given either by $f(k_i)$ when $k_i\in K_i$ is the generator or by $[\sigma(\eta_j)]$ when $\eta_j$ represents the generator $h_j\in H_j$.
\end{proof}

\section{Maps out of suspensions II}

\subsection*{Notation}

From now on, for $q\geq 1$, we denote $\Gp{\then}{\theq}=[I^\theq\times X,\Pnew]^{(\partial I^\theq\times X)\cup(I^\theq\times A)}_B$.

\subsection*{Proof of Theorem~\ref{thm:main_ptb}}

We are now ready to finish the proof of Theorem~\ref{thm:main_ptb}. We will formalize its statement in the following claim:
\begin{enumerate}[topsep=2pt,itemsep=2pt,parsep=2pt,labelindent=0.5em,leftmargin=*,label=\textbf{(poly)${}_n$}]
\renewcommand{\theenumi}{\textbf{(poly)}}
\renewcommand{\labelenumi}{\textbf{(poly)${}_\then$}}
\item\label{cl:poly}
	It is possible to equip $\Gp{\then}{\theq}$ with a structure of a fully effective polycyclic group.
\end{enumerate}
Since $\Gp{\then}{\theq}$ for $(X,A)$ is a $\Gp{\then}{1}$ for $(I^{\theq-1},(\partial I^{\theq-1}\times X)\cup(I^{\theq-1}\times A))$, it would be enough to restrict to the case $\theq=1$. This special case also implies Theorem~\ref{thm:main_ptb}:

\begin{proof}[Proof of Theorem~\ref{thm:main_ptb} from \clpoly{\then}]
Let $\then=\dim(I\times X)=1+\dim X$. Then
\[[\Sigma_BX,Y]^{\Sigma_BA}_B\cong[I\times X,Y]^{(\partial I\times X)\cup(I\times A)}_B\cong[I\times X,\Pnew]^{(\partial I\times X)\cup(I\times A)}_B=\Gp{\then}{1}\]
and the last term is computable by \clpoly{\then}.
\end{proof}

The remainder of this section is devoted to the proof of \clpoly{\then}. We observe that
\[[I^\theq\times X,\Ln]^{(\partial I^\theq\times X)\cup(I^\theq\times A)}_B\cong H^{\then}(I^\theq\times X,(\partial I^\theq\times X)\cup(I^\theq\times A);\pi_\then)\cong H^{\then-\theq}(X,A;\pi_\then)\]
and similarly $[I^\theq\times X,\Kn]^{(\partial I^\theq\times X)\cup(I^\theq\times A)}_B\cong H^{\then+1-\theq}(X,A;\pi_\then)$.

\begin{proof}[Proof of \clpoly{\then-1}${}+{}$\clnull{\then-1}${}\Ra{}$\clpoly{\then}]
The computation in \cite{aslep} was based on fully effective abelian groups. With the notion of a fully effective polycyclic group at hand, we may proceed in the same way. Namely, the group $\Gp{\then}{\theq}$ is semi-effective by Proposition~\ref{p:semi_effective}. The exact sequence \eqref{e:les} applied to $(I^\theq\times X,(\partial I^\theq\times X)\cup(I^\theq\times A))$ instead of $(X,A)$ reads
\[\Gp{\then-1}{\theq+1}\xlra{\partial}H^{\then-\theq}(X,A;\pi_\then)\xlra{j_*}\Gp{\then}{\theq}\xlra{p_{\then*}}\Gp{\then-1}{\theq}\xlra{\knst}H^{\then+1-\theq}(X,A;\pi_\then)\]
and induces a short exact sequence
\[\xymatrix{
0 \ar[r] & \coker\partial \ar[r]_-{\jnst} & \Gp{\then}{\theq} \ar[r]_-{p_{n*}} \sect{t} & \ker k_{\then*} \ar[r] \sect{\sigma} & 0
}\]
with the first term fully effective abelian by Lemma~\ref{l:ker_coker} and the last term fully effective polycyclic by Proposition~\ref{p:poly_ker}; both claims use \clpoly{\then-1}.

For the application of Proposition~\ref{p:poly_extension}, we need to provide algorithms for the two indicated sections. The section $\sigma$ is defined on the level of representatives (on which it depends) by mapping a partial diagonal $f\colon I^\theq\times X\ra\Pold$ to an arbitrary lift $\widetilde{f} \colon I^\theq\times X\ra\Pnew$ of $f$ that is zero on $(\partial I^\theq \times X) \cup (I^\theq \times A)$. The computation of $\widetilde{f}$ is taken care of by Proposition~\ref{prop:lift_ext_one_stage}.

For the construction of the partial inverse $t$ on $\im\jnst=\ker\pnst$, let $f\colon I^\theq\times X\ra\Pnew$ be a diagonal such that its composition with $\pn\colon\Pnew\ra\Pold$ is homotopic to zero. Then we can compute such a nullhomotopy $h$ by \clnull{n-1}. Using Proposition~\ref{prop:homotopy_lifting}, we lift it along $\pn$ to a homotopy from some $f'$ to $f$. Since $\pn f'=\oold$, the image of $f'$ lies in $\Ln$ and we may set $t([f])=[f']$.
\end{proof}

\begin{rem}
It is possible to organize the computation of $\Gp{\then}{1}$ in Theorem~\ref{thm:main_ptb} in such a way that the algorithm only accesses a fully effective polycyclic structure on $\Gp{\them}{1}$, $\them\leq\then$, and generators of $\Gp{\them}{2}$, $\them<\then$. This is because one may easily compute generators of $\Gp{\then}{\theq}$ from those of $\ker\knst$ and $\coker\partial$. Now, $\coker\partial$ is generated by the images of generators of $H^{\then-\theq}(X,A;\pi_\then)$ and for $q\geq 2$, it is not too difficult to compute a set of generators of $\ker\knst$ from a set of generators of the \emph{abelian} group $\Gp{\then-1}{\theq}$.%
\footnote{
	Pretend that the domain is free abelian on the provided generators and compute the generators of the kernel in this situation -- they generate the kernel even when the domain is not free abelian. The same procedure for non-abelian groups may easily lead to infinitely generated groups.
}
Thus, for $\theq\geq 2$, the computation of generators of $\Gp{\then}{\theq}$ can be executed by induction on $\then$ while $\theq$ is kept fixed.
\end{rem}

\section{Polynomiality}

\subsection*{Proof of the polynomiality claim of Theorem~\ref{thm:main}}
The definitions and most of the ingredients are contained in \cite[Section~8]{aslep}. First, we discuss the algorithmic aspects of polycyclic groups, i.e.\ Propositions~\ref{p:poly_ker} and~\ref{p:poly_extension}. For this purpose, we introduce the notion of a \emph{family of fully effective polycyclic groups}. This is a collection of polycyclic groups $(G(p))_{p\in\sfP}$ represented on sets $(\mcG(p))_{p\in\sfP}$ together with the following algorithms:
\begin{itemize}[topsep=2pt,itemsep=2pt,parsep=2pt,leftmargin=\parindent]
\item
	input: $p\in\sfP$, $\gamma\comma\delta\in\mcG(p)$; output: representatives of $0\comma[\gamma]+[\delta]\comma{-[\gamma]}\in G(p)$;
\item
	input: $p\in\sfP$; output: the list $\gamma_1,\ldots,\gamma_r$ of representatives of the generators $g_1,\ldots,g_r$ of $G(p)$ and the list $q_1,\ldots,q_r$ of the orders of $G_i/G_{i-1}$, as in Definition~\ref{d:poly};
\item
  input: $p\in\sfP$, $\gamma\in\mcG(p)$; output: the list of coefficients $z_1,\ldots,z_r$ such that $[\gamma]=z_1g_1+\cdots+z_rg_r$ holds in $G(p)$, each $z_i$ unique within $\bbZ/q_i$.
\end{itemize}
A \emph{polynomial-time family} is one for which the running times of all these algorithms are bounded by a polynomial in the size of the input. Similarly, a (polynomial-time) family of semi-effective groups consists only of the algorithms in the first point.

We view Proposition~\ref{p:poly_ker} as a construction (i.e.\ a mapping, but in the situation where both sides consist of computational structures, see \cite{aslep})
\[\xymatrix{
\left\{\parbox{\widthof{computable group homomorphisms}}{computable group homomorphisms $f\colon G\to H$ with $G$ fully effective polycyclic, $H$ fully effective abelian}\right\} \ar[r] & \left\{\parbox{\widthof{fully effective polycyclic groups}}{fully effective polycyclic groups}\right\},
}\]
sending $f\colon G\to H$ to $\ker f$, and claim that it is polynomial-time. Since the computation of the generators $g_1,\ldots,g_r$ takes polynomial time, $r$ is bounded by this polynomial. Thus, the inductive computation of the generators of $\ker f$ takes a polynomial number of steps. Each step is performed in time that is bounded by a fixed polynomial. Thus, the total running time is also polynomial. The same holds for the projections $K_i\to K_i/K_{i-1}\cong\bbZ/(q_it_i^{-1})$.

Proposition~\ref{p:poly_extension} is seen to be a polynomial-time construction defined on short exact sequences of semi-effective groups equipped with the indicated set-theoretic sections and with outer terms fully effective polycyclic; it takes values in fully effective polycyclic groups.

We will now formalize the algorithms of \clnull{\then} and \clpoly{\then}. First, it is possible to compute from $\varphi\colon Y\to B$ in polynomial time the parameters of the Moore--Postnikov tower $P_n$, giving a polynomial-time mapping $\Map \ra \MPS_n$; details on the parameter set $\MPS_n$ (Moore--Postnikov system) as well as $\PMPS_n$ (pointed Moore--Postnikov system) can be found in \cite{aslep}. We write $(S,o)\in\PMPS_n$ with $S\in\MPS_n$ and $o$ denoting the zero section. Proposition~\ref{p:semi_effective} provides a polynomial-time family
\[\xymatrix@R=1pc{
\rightbox{\Gpse{n}{\theq}\colon\PairPMPS_n}{} \ar@{~>}[r] & \leftbox{}{\{\textrm{semi-effective groups}\}} \\
\rightbox{(X,A,\beta,S,o)}{} \ar@{|->}[r] & \leftbox{}{[I^\theq\times X,\Pnew]^{(\partial I^\theq\times X)\cup(I^\theq\times A)}_B=\Gp{\then}{\theq}}
}\]
defined on pairs $(X,A)$ over $B$ together with $(S,o)$ parametrizing a pointed Moore--Postnikov system over $B$. The nullhomotopy algorithm is a polynomial-time construction
\[\xymatrix{
\left\{\parbox
	{\widthof{$(P_\then,o)$ a pointed Moore--Postnikov system over $B$,}}
	{$(P_\then,o)$ a pointed Moore--Postnikov system over $B$, $f\colon I^\theq\times X\to\Pnew$ over $B$ that is zero on $(\partial I^\theq\times X)\cup(I^\theq\times A)$, generators of each $\Gp{1}{\theq+1},\ldots,\Gp{\then-1}{\theq+1}$}
\right\} \ar[r] &
\left\{\parbox
	{\widthof{$h\colon I^{\theq+1}\times X\to\Pnew$}}
	{$h\colon I^{\theq+1}\times X\to\Pnew$}
\right\}\cup\{\bot\};
}\]
either it gives $\bot$ if $f$ is not nullhomotopic or it computes a nullhomotopy of $f$. Assuming that $\Gpse{\then-1}{\theq}$ has been lifted to a polynomial-time family $\Gpfe{\then-1}{\theq}$ of fully effective polycyclic groups, Propositions~\ref{p:poly_ker} and~\ref{p:poly_extension} then lift $\Gpse{\then}{\theq}$ to a polynomial-time family of fully effective polycyclic groups
\[\xymatrix{
\Gpfe{\then}{\theq}\colon\PairPMPS_\then\times\Gen{1}{\theq+1}\times\cdots\times\Gen{\then-1}{\theq+1} \ar@{~>}[r] & \{\textrm{fully effective polycyclic groups}\},
}\]
where an element of $\Gen{\them}{\theq}$ is a list of generators of $\Gp{\them}{\theq}$; of course, there are some compatibility constraints between $\PairPMPS_\then$ and the $\Gen{\them}{\theq}$. The parameters in $\Gen{\then}{\theq}$ are computed recursively using the fully effective $\Gpfe{\then}{\theq}$ or, for $\theq\geq 2$, using the parameters from $\Gen{\them}{\theq}$, $\them<\then$, see the remark at the end of the previous section.
\qed

\subsection*{Acknowledgement}

We are grateful to Martin \v{C}adek for carefully reading the paper and for his useful comments and suggestions.

\end{document}